\documentclass[twoside, 12pt, a4paper, reqno]{amsart}
\usepackage[T1]{fontenc}
\usepackage[utf8]{inputenc}
\usepackage{lmodern}
\usepackage[pdftex]{graphicx}
\graphicspath{{./Images/1/}{./Images/}}
\usepackage{float}
\usepackage{url}
\usepackage{soul}
\usepackage[alphabetic,nobysame,abbrev]{amsrefs}
\usepackage{microtype}
\usepackage[all]{xy}
\usepackage{amssymb,relsize,datetime2}
\usepackage{booktabs,multicol}
\usepackage[pdftex, dvipsnames]{xcolor}
\usepackage{caption}
\usepackage{subcaption}
\usepackage{transparent}
\usepackage{ifpdf}
\usepackage{import}
\usepackage{overpic}
\usepackage{multirow}
\usepackage{tikz-cd}
\usepackage{adjustbox}
\usepackage{parskip}
\usepackage{xpatch}
\usepackage{comment}
\usepackage[normalem]{ulem}
\usepackage[breaklinks,%
   pdftex,%
   final,%
   colorlinks=true,%
   linkcolor=NavyBlue,%
   citecolor=blue,
   filecolor=NavyBlue,%
   menucolor=NavyBlue,%
   urlcolor=NavyBlue,%
   bookmarks=true,%
   bookmarksdepth=2,%
   bookmarksnumbered=true,%
   bookmarksopen=true,%
   bookmarksopenlevel=2,%
   unicode
]{hyperref}
\usepackage[capitalize,noabbrev]{cleveref}

\usepackage{array}
\makeatletter
\def\env@dmatrix{\hskip -\arraycolsep
  \let\@ifnextchar\new@ifnextchar
  \extrarowheight=2ex
  \array{*\c@MaxMatrixCols{>{\displaystyle}c}}}

\makeatother

\makeatletter
\newtheorem*{rep@theorem}{\rep@title}
\newcommand{\newreptheorem}[2]{%
\newenvironment{rep#1}[1]{%
 \def\rep@title{#2 \ref{##1}}%
 \begin{rep@theorem}}%
 {\end{rep@theorem}}}
\makeatother
\newreptheorem{theorem}{Theorem}
\newreptheorem{lemma}{Lemma}
\newreptheorem{proposition}{Proposition}
\newreptheorem{corollary}{Corollary}		

\crefname{question}{Question}{Questions}

\crefname{alpharesults}{Theorem}{Theorems} % new
\crefname{alphatheorem}{Theorem}{Theorems}

\usepackage{enumitem}

\usepackage{listings} %source code

\usepackage{footnotebackref}
\makeatletter
\xpretocmd{\@adminfootnotes}{\let\@makefntext\BHFN@OldMakefntext}{}{}
\renewcommand\@makefntext[1]{%
  \@ifundefined{@makefnmark}
    {}
    {%
     \renewcommand\@makefnmark{%
       \mbox{%
         \textsuperscript{%
           \normalfont
           \hyperref[\BackrefFootnoteTag]{\@thefnmark}%
         }%
       }\,%
     }%
     \BHFN@OldMakefntext{#1}%
  }%
}
\makeatother

\newtheorem{theorem}{Theorem}%[section]

\newtheorem{lemma}[theorem]{Lemma}

\newtheorem{question}[theorem]{Question}

\newcounter{alpharesults}

\theoremstyle{definition}

\newtheorem*{definition-nonum}{Definition}

\theoremstyle{remark}

\newtheorem{remark}[theorem]{Remark}

\numberwithin{equation}{section}

\newcommand{\ie}{i.e.~}

\newcommand{\eg}{e.g.~}

\newcommand{\Q}{\mathbb{Q}}
\newcommand{\Z}{\mathbb{Z}}

\definecolor{amaranth}{rgb}{0.75, %0.9
0.17, 0.31} %dark red
\definecolor{carrotorange}{rgb}{0.93, 0.57, 0.13} %orange
\definecolor{citrine}{rgb}{0.89, 0.82, 0.04} %dark yellow
\definecolor{dartmouthgreen}{rgb}{0.05, 0.5, 0.06} %green
\definecolor{teal}{rgb}{0.0, 0.5, 0.5} %teal
\definecolor{ballblue}{rgb}{0.13, 0.67, 0.8} %blue
\definecolor{ceruleanblue}{rgb}{0.16, 0.32, 0.75} %deeper blue
\definecolor{amethyst}{rgb}{0.6, 0.4, 0.8} %purple
\definecolor{amber}{rgb}{1.0, 0.75, 0.0} %amber
\definecolor{burlywood}{rgb}{0.87, 0.72, 0.53} %beigebrown
\definecolor{dogwoodrose}{rgb}{0.84, 0.09, 0.41} %dogwoodrose
\usepackage{mathtools}
\newcommand{\defeq}{\vcentcolon=}

\AtBeginDocument{%
   \def\MR#1{}
}

\title{Braid positive surgery diagrams}
\author{Marc Kegel}
\address{Universidad de Sevilla, Dpto.\ de Álgebra,
Avda.\ Reina Mercedes s/n,
41012 Sevilla}
\email{\href{mailto:mkegel@us.es}{mkegel@us.es}, \href{mailto:kegelmarc87@gmail.com}{kegelmarc87@gmail.com}}

\author{Paula Truöl}
\address{School of Mathematics and Statistics,
University of Glasgow,
University Place,
Glasgow,
G12 8QQ,
United Kingdom}
\email{
\href{mailto:paula.truol@glasgow.ac.uk}{paula.truol@glasgow.ac.uk},
\href{mailto:paulagtruoel@gmail.com}{paulagtruoel@gmail.com}}

\date{\today}

\begin{document}

%\def\subjclassname{\textup{2020} Mathematics Subject Classification}
%\expandafter\let\csname subjclassname@1991\endcsname=\subjclassname
%\subjclass{
%57K10% Knot theory
%}
\makeatletter
\@namedef{subjclassname@2020}{%
	\textup{2020} Mathematics Subject Classification}
\makeatother%For 2020

% 57R65   	Surgery and handlebodies
% 57K10   	Knot theory


\begin{abstract}
In this short note, we prove that every closed, oriented, connected $3$-manifold arises as Dehn surgery along a braid positive link.
\end{abstract}
\maketitle

\section{Introduction}
A classical result by Lickorish and Wallace~\cites{lickorish,wallace} states that every closed, oriented, connected $3$-manifold $M$ arises as \emph{Dehn surgery} along a link $L$ in $S^3$. That is, $M$ can be obtained by removing an open tubular neighbourhood $\nu L$ of $L$ from $S^3$ and gluing back to each torus boundary component of $S^3 \setminus \nu L$ a solid torus by an orientation-preserving diffeomorphism of its boundary. 
Our motivation in this note stems from the following type of question.

\begin{question}\label{question}
    Let $\mathcal{L}$ be a set of links in $S^3$. Can each closed, oriented, connected $3$-manifold be obtained by Dehn surgery along a link $L \in \mathcal{L}$?
\end{question}

Considering the first homology of a $3$-manifold obtained by Dehn surgery directly implies that the number of components of links in $\mathcal{L}$ is unbounded if $\mathcal{L}$ gives a positive answer to Question~\ref{question}. Question~\ref{question} was answered affirmatively for the following sets of links:
\begin{enumerate}
    \item links that are fibered, by \cite{stallings_1978}*{Theorem 3},
    \item links that are quasipositive, by \cite{Rudolph:constr_II}*{Proposition 4},
    \item links that are alternating and hyperbolic, by \cite{adams}*{p.~126},
    \item chainmail links, by \cite{polyak}; see also \cite{agol2023chainmaillinkslspaces}*{Theorem 2.2},
    \item generalized $L$-space links, by \cite{agol2023chainmaillinkslspaces}*{Corollary 4.3}.\footnote{Note that even though $L$-space knots are fibered, not every (generalized) $L$-space link is fibered; see \cite{liu}*{Example 3.9}.}
\end{enumerate}

In contrast to these results, not every closed, oriented, connected $3$-manifold can be obtained by Dehn surgery along a torus link since Dehn surgeries along torus links are never hyperbolic \cites{moser,surgery}. Moreover, the answer to Question~\ref{question} is negative for the set of Montesinos links by~\cite{Ichihara_2009} and thus also for pretzel links. 
%That the answer to Question~\ref{question} is negative for 2-bridge links follows directly since a $2$-bridge link has at most two components.
%The answer to Question~\ref{question} is also negative for 2-bridge links, since these have at most two components.
Question~\ref{question} also has a negative answer for 2-bridge links since these have at most two components.

In this short note, we show the following.

\begin{theorem}\label{thm:main}
	Every closed, oriented, connected $3$-manifold~$M$ can be obtained by Dehn surgery along a braid positive link $L$ in $S^3$. Moreover, $L$ can be chosen to have at most $s(M) + 1$ many components.
\end{theorem}

Here, we denote by~$s(M)$ the \emph{surgery number} of~$M$, which is the minimal number of components of any link in $S^3$ needed to describe~$M$ as Dehn surgery along that link, see for example~\cite{auckly}. Moreover, recall that a link is called \emph{braid positive} if it arises as the closure of a positive braid, \ie a braid represented by a finite product of the Artin generators of the braid group~\cite{artin_1925} (see Section \ref{sec:braids}).

Braid positive links are fibered~\cite{stallings_1978}, positive and strongly quasipositive, so we can also add \eg the set of positive, fibered links or the set of strongly quasipositive, fibered links to the above list.

\subsection*{Conventions}
All $3$-manifolds are assumed to be smooth and oriented, and are considered up to orientation-preserving diffeomorphism. For us, a link in a $3$-manifold $M$ is a smooth, closed $1$-dimensional submanifold of $M$. Surgery coefficients (see Section~\ref{sec:surgery}) are measured with respect to the Seifert framing. 

\subsection*{Acknowledgements} This project started at the conference on \textit{Trisections and related topics}, October 2025, CIRM, Marseille, and was finished at the workshop on \textit{Knotted surfaces in $4$-manifolds}, December 2025, University Regensburg. We thank Stefan Friedl for helpful comments on a first draft.

\subsection*{Grant support}
MK is supported by a Ram\'on y Cajal grant (RYC2023-043251-I) and PID2024-157173NB-I00 funded by MCIN/AEI/10.13039/ 501100011033, by ESF+, and by FEDER, EU; and by a VII Plan Propio de Investigación y Transferencia (SOL2025-36103) of the University of Sevilla. PT acknowledges support by the Swiss National Science Foundation Postdoc.Mobility fellowship 230329, and she would like to thank the Max Planck Institute for Mathematics in Bonn and the University of Glasgow for their hospitality and support.

\subsection*{Competing interests}  The authors declare none.

\section{Surgery diagrams and Rolfsen twists}\label{sec:surgery}
The two main ingredients in the proof of Theorem~\ref{thm:main} are the so-called Rolfsen twist~\cite{rolfsen:rational_calc} and Garside's theory of braids~\cite{GARSIDE}. To make this note self-contained, we will explain Rolfsen twists in this section and a direct consequence of Garside's work needed in our proof in Section~\ref{sec:braids}. Readers who are already familiar with these concepts can skip ahead to Section~\ref{sec:proof} for the proof of Theorem~\ref{thm:main}. 

First, let us briefly review the basics of Dehn surgery. For more details, the reader is referred, for example, to~\cite{rolfsen_2003}. Let $L = L_1 \cup \dots \cup L_m$ be a link in $S^3$ of $m$ components. The boundary of~$S^3 \setminus \nu L$ is orientation-preservingly diffeomorphic to a disjoint union of~$m$ tori~$\bigcup_{j=1}^m S^1 \times S^1$. The diffeomorphism type of the $3$-manifold obtained by Dehn surgery along~$L$, \ie by gluing $m$ solid tori $\bigcup_{j=1}^m S^1 \times D^2$ to $S^3 \setminus \nu L$, is determined by the image of $\bigcup_{j=1}^m \{1\} \times \partial D^2$ under an orientation-preserving diffeomorphism 
$$\bigcup_{j=1}^m \partial (S^1 \times D^2 )\longrightarrow S^3 \setminus \nu L.$$ 
Given a meridian $\mu_j$ and a Seifert longitude $\lambda_j$ of the link component~$L_j$, the image curves in $\partial (\overline{\nu L_j})$ can be expressed as $p_j \mu_j + q_j \lambda_j$ for coprime integers~$p_j, q_j$ for every $j \in \{1, \dots, m\}$. The quotient $$p_j/q_j \in \Q \cup \{\infty\}$$ is called the \emph{surgery coefficient} of the link component~$L_j$. 
By the above discussion, the diffeomorphism-type of the $3$-manifold obtained by Dehn surgery on the link $L$ is determined by a diagram of $L$ where each component~$L_j$ of~$L$ is labelled by $p_j/q_j$. We call~$L$ together with the labellings $p_j/q_j$, $i\in\{1,\dots,m\}$, a \emph{surgery diagram}. 
%Let us remark that a surgery with surgery coefficient $\infty=1/0$ identifies the meridian of the attached solid torus with the meridian of the surgery curve, so the gluing simply glues in a solid torus in exactly the same way as the removed tubular neighbourhood. Thus a component with surgery coefficient $\infty=1/0$ can be removed from a surgery diagram without changing the diffeomorphism type of the surgered manifold.

Let us remark that a link component $L_j$ with surgery coefficient $\infty=1/0$ means that we identify the meridian of the attached solid torus with the meridian of~$L_j$. The corresponding surgery simply glues in a solid torus in exactly the same way as the removed tubular neighbourhood of $L_j$. Therefore, a component with surgery coefficient $\infty=1/0$ can be removed from a surgery diagram without changing the diffeomorphism type of the surgered manifold.

Let us now turn to Rolfsen twists. To that end, let $L= L_1 \cup \dots \cup L_m$ be a surgery diagram with an unknotted component, say $L_1$, with surgery coefficient $p_1/q_1$. For any integer $N\in\Z$, an \textit{$N$-fold Rolfsen} twist along~$L_1$ creates out of $L$ a new surgery diagram $L'$ as follows:
\begin{itemize}
	\item we cut the exterior of $L_1$, which is a solid torus, along a Seifert disk of $L_1$ and re-glue it with an $N$-fold twist (where $N=1$ corresponds to one right-handed full twist)
    to get a new link $L'= L'_1 \cup \dots \cup L'_m$ (whose components are in natural bijection with the components of $L$), 
	\item the surgery coefficient of $L'_1$ is $\frac{1}{N+\frac{q_1}{p_1}}$ and the surgery coefficients of the other components $L'_j$ %, $j \neq 1$, 
    are given by $\frac{p_j}{q_j}+N\operatorname{lk}(L_j,L_1)^2$, where $\operatorname{lk}$ denotes the linking number for a choice of orientation of $L$. 
\end{itemize} 
We refer to Figure~\ref{fig:Rolfsen_twist} for a visualization.

\begin{lemma}[Rolfsen twist~\cite{rolfsen:rational_calc}]
	If $L$ and $L'$ are two surgery diagrams that differ by a Rolfsen twist, then the $3$-manifolds obtained by Dehn surgery along $L$ and $L'$, respectively, are orientation-preservingly diffeomorphic. 
\end{lemma}

\begin{proof}
	The process of cutting the exterior of $L_1$ along a Seifert disk of~$L_1$ and re-glueing it with an $N$-fold twist describes a diffeomorphism of the exterior of $L$ to the exterior of $L'$. It is straightforward to check that this diffeomorphism acts on the boundary tori as
	\begin{align*}
		\mu_1 &\longmapsto \mu'_1+N\lambda'_1\\
		\lambda_1 &\longmapsto \lambda'_1\\
		\mu_j &\longmapsto \mu'_j\\
		\lambda_j &\longmapsto N\operatorname{lk}(L_j,L_1)^2\mu'_j+\lambda'_j
	\end{align*}
and thus the surgery coefficients transform as claimed. 
\end{proof}

\begin{figure}[htbp] 
	\centering
    \vspace{1em}
   \begin{overpic}[width=0.65\textwidth]{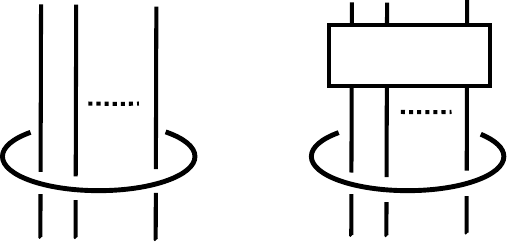}
   \put(48,25){$\cong$}
   \put(78,35){$N$}
   
   \put(39,10){$L_1$}
   \put(55,10){$L'_1$}
   
   \put(5,50){$L_2$}
   \put(65,50){$L'_2$}
   \put(13,50){$L_3$}
   \put(75,50){$L'_3$}
   \put(29,50){$L_m$}
   \put(90,50){$L'_m$}

   \put(-6,10){$\frac{p_1}{q_1}$}
   \put(100,10){$\frac{1}{N+\frac{q_1}{p_1}}$}

   \put(5,-5){$\frac{p_2}{q_2}$}
   \put(60,-5){$\frac{p_2}{q_2}+N\operatorname{lk}^2$}

   \put(13,-5){$\frac{p_3}{q_3}$}
   %\put(70,-5){$\frac{p_3}{q_3}+N\operatorname{lk}^2$}

   \put(29,-5){$\frac{p_m}{q_m}$}
   \put(90,-5){$\frac{p_m}{q_m}+N\operatorname{lk}^2$}
 \end{overpic}
  \vspace{1em}
    \caption{An $N$-fold Rolfsen twist.}
    \label{fig:Rolfsen_twist}
\end{figure}
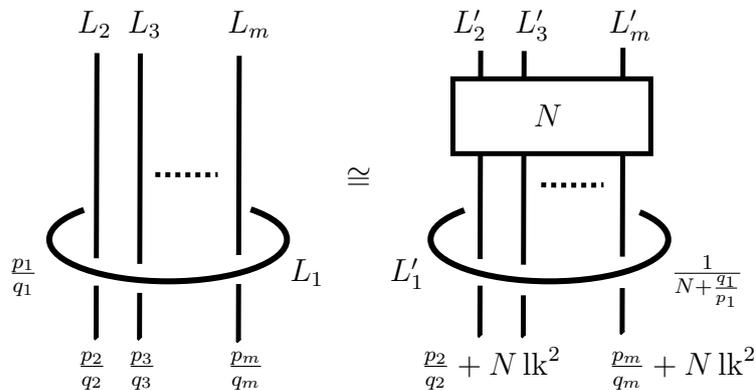

\section{Garside theory}\label{sec:braids}

The second key ingredient in our proof of Theorem~\ref{thm:main} is Garside's theory of braids~\cite{GARSIDE}. A \textit{braid} on $k$ strands is a collection of $k$ disjoint, properly embedded arcs in the cylinder $[0,1]\times D^2 $ that run monotonically in the $[0,1]$-direction from $k$ fixed points in $\{0\} \times D^2$ to $\{1\} \times D^2$. Braids up to ambient isotopy of $[0,1]\times D^2$ fixing $\{0,1\} \times D^2$ pointwise form a group called the \emph{braid group}~$B_k$. The group operation is defined by stacking braids. Closing a braid by identifying $\{0\} \times D^2$ and $\{1\} \times D^2$, connecting the endpoints of the arcs, and identifying the resulting solid torus with the tubular neighbourhood of an unknot in $S^3$ produces an oriented link in $S^3$, see Figure~\ref{fig:braid}.\footnote{%
%Note that there infinitely many ways to identify a solid torus with the tubular neighborhood of a knot. Here we choose the identification that sends $S^1\times \{*\}$ to the Seifert longitude of the unknot.
Note that there are infinitely many ways to identify a solid torus with the tubular neighborhood of a knot. Here we choose the identification that maps $S^1\times \{0\}$ onto the Seifert longitude of the unknot.
} By Alexander's theorem, every link arises in this way~\cite{alexander}. We refer to \cite{birmanbrendle} for more details on braids and their closures. 

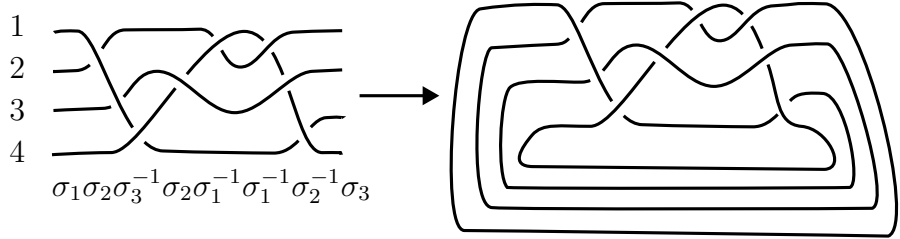
\begin{figure}[htbp] 
\centering
   \begin{overpic}[width=0.88\textwidth]{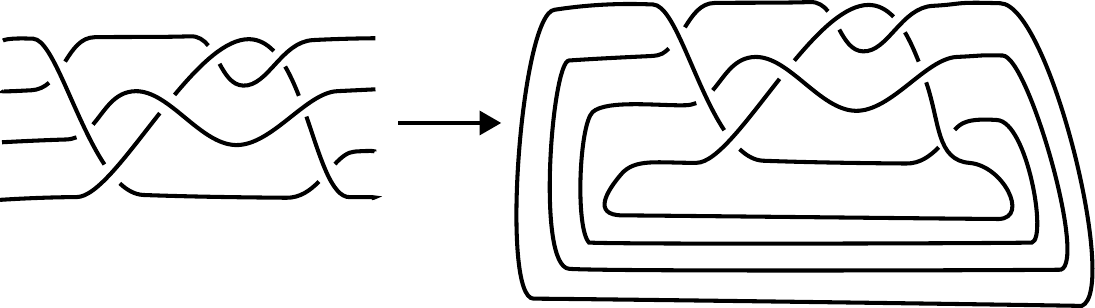}
   \put(-5,24){$1$}
   \put(-5,19){$2$}
   \put(-5,14){$3$}
   \put(-5,9){$4$}
   
   \put(0,5){$\sigma_1 \sigma_2 \sigma_3^{-1} \sigma_2 \sigma_1^{-1} \sigma_1^{-1} \sigma_2^{-1} \sigma_3$}
 \end{overpic}
    \caption{A braid (left) and its closure (right).}
    \label{fig:braid}
\end{figure}

The braid group $B_k$ is generated by the \textit{Artin generators}~\cite{artin_1925} denoted~$\sigma_1, \dots, \sigma_{n-1}$, where $\sigma_i$ corresponds to a positive half–twist exchanging two neighbouring strands with indices $i$ and~$i+1$, see Figure~\ref{fig:braid}. A braid on $k$ strands is \emph{positive} if it can be represented by a word in~$B_k$ that is a product of positive powers of the Artin generators. The closure of a positive braid is called a \emph{braid positive} link. 
We denote by 
$$T_k=(\sigma_1 \dots \sigma_{k-1})^{k}$$ %or $\Delta_{k-1}^2$? Just $\Delta_k$ is bad (contradicts \cite{GARSIDE})
%$$\Delta_k=(\sigma_1 \dots \sigma_{k-1})^{k}$$
a full twist in $B_k$, see for example Figure~\ref{fig:full_twist}. The following lemma is a direct corollary from Garside's work~\cite{GARSIDE}.\footnote{Note that in Garside's notation, we have $T_k = \Delta^2$. % (=\Delta_{k-1}^2)$.
} 
%For completeness, we present a proof.
For completeness and to keep our arguments purely geometric, we present a proof.

\begin{lemma}[{\cite{GARSIDE}}]\label{lem:Garside}
	For each braid $\beta$ on $k$ strands, there is an integer~$N \geq 0$ such that~
    $T_k^{N} \beta$ is a positive braid on $k$ strands. More concretely, we can choose $N$ to be the number of negative Artin generators in a braid word representative of $\beta$.
\end{lemma}

\begin{proof}
We start with two simple observations. First, the full twist $T_k$ in~$B_k$ commutes with any other element $\beta$ of $B_k$, \ie $T_k\beta=\beta T_k$. See Figure~\ref{fig:full_twist} for an illustration. Second, for any Artin generator $\sigma_i$, the braid $T_k\sigma_i^{-1}$ is positive; see Figure~\ref{fig:full_twist} again. 

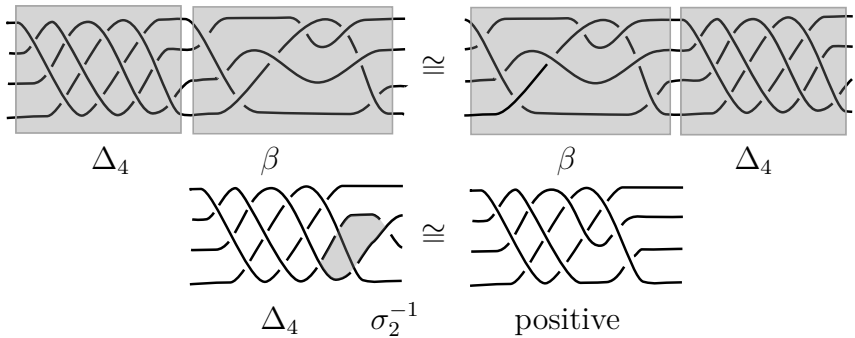
\begin{figure}[htbp] 
	\centering
    \begin{overpic}[width=0.88\textwidth]{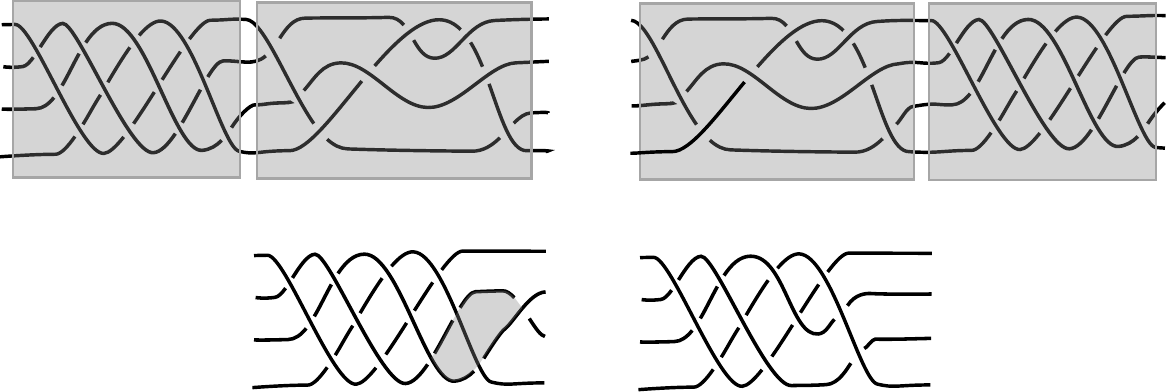}
   \put(10,14){$T_4$} %\Delta_4
   \put(30,14){$\beta$}

   \put(49,25){$\cong$}
 
   \put(65,14){$\beta$}
   \put(86,14){$T_4$} 

   \put(49,5){$\cong$}

   \put(30,-5){$T_4$} 
   \put(43,-5){$\sigma_2^{-1}$}
   \put(60,-5){positive}
 \end{overpic}
 \vspace{1em}
    \caption{The first row shows that a full twist $T_k$ commutes with any other braid $\beta$ on $k$ strands. The second row visualizes that a full twist $T_k$ followed by a negative Artin generator~$\sigma_i^{-1}$ yields a positive braid.}
    \label{fig:full_twist}
\end{figure}

If $\beta$ is a positive braid, we can set~$N = 0$. Otherwise, we write $\beta$ as a word in the Artin generators and distinguish two cases. In the first case, the first letter of $\beta$ is a positive generator, say $\beta=\sigma_i\beta'$. Using the above observations, we can then write
$T_k^N\beta=\sigma_iT_k^N\beta'.$
In the second case, if the first letter of $\beta$ is a negative Artin generator, say~$\beta=\sigma_i^{-1}\beta'$, then we write
\begin{equation*}
    % \Delta_k^N\beta=\Delta_k^{N-1}\beta'=p\Delta_k^{N-1}\beta',
    T_k^N\beta %= \Delta_k^N\sigma_i^{-1}\beta'
    = T_k^{N-1} T_k\sigma_i^{-1}\beta'
    = T_k^{N-1} p \beta' = p T_k^{N-1} \beta',
\end{equation*}
	where $p$ is a positive braid. The claim follows using an induction on the length of the braid word $\beta$.
\end{proof}

\section{Braid positive surgery diagrams}\label{sec:proof}

In this section, we finally prove Theorem~\ref{thm:main}.

\begin{proof}[{Proof of Theorem~\ref{thm:main}}]
Given a closed, oriented, connected $3$-manifold~$M$, let $L$ be a link whose number of components realizes $s(M)$. We can represent $L$ as the closure of a braid $\beta$ on, say, $k$ strands. Figure~\ref{fig:pf}(a) represents a braid diagram $\beta$ of $L$.
Inserting a new component with surgery coefficient $\infty=1/0$ and performing an $N$-fold Rolfsen twist, we obtain the diagram shown in Figure~\ref{fig:pf}(b). Figure~\ref{fig:pf}(c) shows the result of an isotopy of the new link diagram performed on the new component. The link $L^\prime$ in Figure~\ref{fig:pf}(c) arises as the closure of a braid~$\beta^\prime$. Note that~$\beta^\prime$ is a positive braid whenever the braid $\alpha \defeq T_k^{N} \beta$ on $k$ strands is a positive braid. The theorem thus follows from Lemma~\ref{lem:Garside} choosing~$N$ sufficiently large.
\end{proof}

\begin{figure}[htbp] 
\centering
    \begin{overpic}[width=0.88\textwidth]{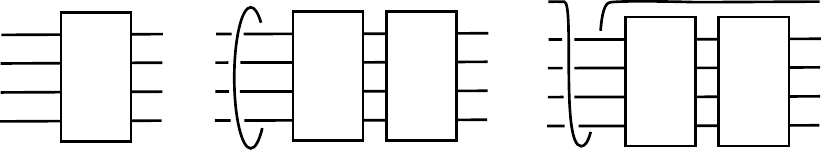}
   \put(11,8){$\beta$}
   \put(8,-4){(a)}

   \put(21,8){$\cong$}
   
   \put(31,20){$\frac{1}{N}$}
   \put(37,8){$T_k^N$}
   \put(50,8){$\beta$}
   \put(42,-4){(b)}

   \put(61.5,8){$\cong$}

   \put(71,20){$\frac{1}{N}$}
   \put(78,8){$T_k^N$}
   \put(91,8){$\beta$}
   \put(81,-4){(c)}
 \end{overpic}
  \vspace{1em}
\caption{Transforming a braid into a positive braid by a single surgery.}
    \label{fig:pf}
\end{figure}
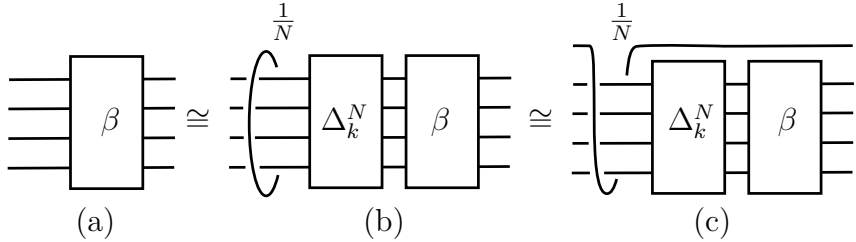

\begin{remark}
    Note that Rudolph's proof for quasipositive links~\cite{Rudolph:constr_II}, in very brief, inserts a new unknotted component for each negative band generator in a band presentation of a braid representing $L$, so it increases both the braid index and the number of components of $L$ potentially much more than the procedure described in our above proof. Moreover, his method does not immediately carry over to obtain braid positive links.
\end{remark}    

\bibliographystyle{alpha}%hamsalpha
\bibliography{biblio}

\end{document}